\documentclass{amsart}
\usepackage{amsmath}
\usepackage{amssymb}
\usepackage{latexsym}
\usepackage{amsfonts}
\usepackage[dvips]{graphicx}
\usepackage{color}
\usepackage{graphicx,epstopdf}

\newtheorem{thm}{Theorem}

\usepackage{float}
\usepackage[section]{placeins}

\newcommand{\p}{\partial}

\renewcommand{\a}{\alpha}
\renewcommand{\b}{\beta}

\renewcommand{\hat}{\widehat}

\def\beq#1#2\eeq{%
        \begin{equation}%
        \label{#1}%
            #2%
        \end{equation}%
    }

\usepackage{color}
\usepackage{graphicx}
\usepackage{epstopdf}

\begin{document}

\title[$\vee$-systems, holonomy Lie algebras and logarithmic vector fields]{$\vee$-systems, holonomy Lie algebras and logarithmic vector fields}

\author{M.V. Feigin}\address{
School of Mathematics and Statistics, 
University of Glasgow, 
15 University Gardens, 
Glasgow G12 8QW, 
UK}
\email{Misha.Feigin@glasgow.ac.uk}

\author{A.P. Veselov}
\address{Department of Mathematical Sciences,
Loughborough University, Loughborough LE11 3TU, UK  and Moscow State University, Moscow 119899, Russia}
\email{A.P.Veselov@lboro.ac.uk}

\maketitle

\begin{abstract}
It is shown that the description of certain class of representations of the holonomy Lie algebra $\mathfrak  g_{\Delta}$ 
associated to hyperplane arrangement $\Delta$ is essentially equivalent to the classification of $\vee$-systems associated to $\Delta.$
The flat sections of the corresponding $\vee$-connection can be interpreted as vector fields, which are both logarithmic and gradient.
We conjecture that the hyperplane arrangement of any $\vee$-system is free in Saito's sense and show this for all known $\vee$-systems and for a special class of $\vee$-systems called harmonic, which includes all Coxeter systems. In the irreducible Coxeter case the potentials of the corresponding gradient vector fields turn out to be Saito flat coordinates, or their one-parameter deformations. We give formulas for these deformations as well as for the potentials of the classical families of harmonic $\vee$-systems. 
\end{abstract}

\bigskip

\section{Introduction}

The theory of the hyperplane arrangements has a rich history and is known to be related to many areas of mathematics: combinatorics, topology, singularity theory and classical analysis, see the Introduction in \cite{OT}.

It was the link with the singularity theory, which was the motivation to study the logarithmic vector fields for K. Saito, who introduced an important notion of a free arrangement \cite{Saito}. However, in spite of substantial work done since then in this area (mainly by Terao and his school, see e.g. recent paper \cite{A}) there is not much known about how large the class of free arrangements is. As Orlik and Terao wrote in \cite{OT}: ``It is an enduring mystery of the subject just what makes an arrangement free,'' which is still valid now. In particular, there is still not much evidence in favour of Terao's conjecture that freeness is a combinatorial property and depends only on the intersection lattice \cite{OT}. 

One of our main messages is that this problem may be closely related to the classification of the logarithmic Frobenius structures and $\vee$-systems, which is also a largely open problem, see the most recent discussion in \cite{SchV}. 
The $\vee$-systems are special finite covector configurations introduced in \cite{V1,V2} in relation with certain class of solutions of the generalized Witten-Dijkgraaf-Verlinde-Verlinde (WDVV) equations, playing a fundamental role in 2D topological  field theory, $N=2$ SUSY Yang-Mills theory and the theory of Frobenius manifolds introduced by Dubrovin \cite{DVV, D1, MMM}. 
They can be defined as follows.

Let $V$ be a complex vector space and  $\mathcal{A}\subset V^*$ be a
finite set of pairwise non-collinear vectors in the dual space $V^*$ (covectors) spanning $V^*$. To such a set one can
associate the following {\it canonical form} $G_{\mathcal A}$ on
$V$:
\begin{equation}
\label{GEC} 
G_{\mathcal A}(x,y)=\sum_{\a\in\mathcal{A}}\a(x)\a(y),
\end{equation}
where $x,y\in V$. Let us assume that this form is non-degenerate and thus
establishes the isomorphism
$$
\varphi_{\mathcal A}: V \rightarrow V^*.
$$
Let $\a^\vee = \varphi_{\mathcal A}^{-1}(\alpha)$ be the corresponding inverse image of $\alpha \in \mathcal A.$
Note that because of the choice of the canonical form $\alpha^\vee$ is a complicated function of all $\alpha \in \mathcal A$.

The system $\mathcal{A}$ is called $\vee$-{\it system} if the following $\vee$-{\it conditions}
\begin{equation}
\label{vee}
\sum\limits_{\beta \in \Pi \cap \mathcal {A}}
\beta(\alpha^\vee)\beta^\vee=\nu \alpha^{\vee}
\end{equation}
are satisfied for any $\alpha \in \mathcal{A}$ and any two-dimensional plane $\Pi \subset V^*$ containing $\alpha$ and some $\nu$, which may depend on $\Pi$ and $\alpha.$ If $\Pi$ contains more than 2 covectors then (\ref{vee}) imply that $\nu$ does not depend on $\alpha \in \Pi$ 
and 
\begin{equation}
\label{vee1} 
\sum\limits_{\beta \in \Pi \cap \mathcal {A}}
\beta^\vee \otimes \beta |_{\Pi} = \nu(\Pi) Id.
\end{equation}
If $\Pi$ contains only two covectors from $\mathcal A$, say $\alpha$ and $\beta,$ then (\ref{vee}) imply that 
\begin{equation}
\label{vee2} 
G_{\mathcal A}(\alpha^{\vee}, \beta ^{\vee})=0.
\end{equation}

The examples of $\vee$-systems include all two-dimensional systems, 
Coxeter systems and the so-called deformed root systems \cite{MG, SV,V1}, 
but the full classification is still an open problem (see the latest results in \cite{FV, FV2, LST, SchV}).
The combinatorial (or matroidal) structure of all known $\vee$-systems is quite special, 
but there are no general results known so far.

In this paper we aim to link this problem with Saito's theory of logarithmic vector fields and free arrangements \cite{Saito} and with Kohno's theory of holonomy Lie algebras and logarithmic connections \cite{Kohno2}.
 

For any finite set of non-collinear covectors $\mathcal A \subset V^*$ one can consider the {\it associated arrangement} 
of complex hyperplanes $\Delta=\Delta_\mathcal A:=\cup_{\alpha\in \mathcal A}H_\alpha$ in $V$ given by $\alpha(x)=0, \, \alpha \in \mathcal A$ and the corresponding {\it holonomy Lie algebra} $\mathfrak g_{\Delta}$ with generators $\{t_\alpha\}_{\alpha\in\mathcal A}$ and the relations
\begin{equation}
\label{t_A}
[t_\alpha, \sum_{\beta\in \mathcal A\cap \Pi}t_\beta]=0\ ,\quad \alpha\in \mathcal A\cap \Pi,
\end{equation}
where $\Pi$ is any two dimensional subspace of $V^*$ (see Kohno \cite{Kohno, Kohno2}).  
This Lie algebra coincides with the Lie algebra of the unipotent completion of the
fundamental group of the corresponding complement $\Sigma=V\setminus \Delta$  \cite{Kohno}. Its enveloping algebra is the quadratic dual of the
cohomology algebra $H^*(\Sigma, \mathbb C)$ in the cases when the
latter is quadratic \cite{Yuzvinsky}. The relations (\ref{t_A}) are equivalent to the flatness of the universal logarithmic connection \cite{Kohno2}
\begin{equation}\label{KZK}
\nabla_\xi = \partial_\xi - \kappa \sum_{\alpha \in \mathcal A} \frac{\alpha(\xi)}{\alpha(x)}t_{\alpha}, \,\,\, \xi \in V, x \in \Sigma.
\end{equation}

In particular, for the standard arrangement of hyperplanes $H_{ij}$ in $\mathbb C^n$ given by $z_i-z_j=0, \, 1\leq i<j\leq n$ we have the Kohno-Drinfeld Lie algebra $\mathfrak t_n$ with generators $t_{ij}=t_{ji}, \, 1\leq i<j\leq n$ and relations \cite{Kohno}
\begin{equation}\label{t_KD}
[t_{ij}, t_{kl}]=0, \quad [t_{ij}, t_{ik}+t_{jk}]=0
\end{equation}
for all distinct $i,j,k,l$.

The first result of this paper is a one-to-one correspondence between certain linear representations of holonomy Lie algebras and $\vee$-systems (see Theorem 1 below). It is essentially a reformulation of the known equivalence of the  $\vee$-conditions and the flatness of the corresponding $\vee$-{\it connection} \cite{V2}
\begin{equation}\label{veecon}
\nabla^\vee_\xi=\partial_\xi - \kappa\sum_{\alpha\in {\mathcal A}}\frac{\alpha(\xi)}{\alpha(x)}\alpha^\vee\otimes \alpha, 
\end{equation}
where $\xi \in V, \, x \in \Sigma$ and $\kappa \in \mathbb C$ is a parameter.
Similar result was also pointed out recently by Arsie and Lorenzoni in \cite{Arsie}. 

By identifying $T_x\Sigma$ with $V$ we can view the flat sections of the $\vee$-connection 
\begin{equation}\label{KZ0}
\nabla^\vee_\xi \psi =0, \quad \psi, \xi \in V, \, x \in \Sigma
\end{equation}
as the vector fields on $\Sigma$, which are parallel with respect to $\nabla^\vee_\xi$ ($\vee$-parallel vector fields).
The monodromy of the system (\ref{KZ0}) gives a linear representation of the corresponding fundamental group $\pi_1(\Sigma)$ in $V.$

Important examples of $\vee$-systems are the following classical series found in \cite{CV}:
\begin{equation}\label{an}
A_{n}(c)=\left\{\sqrt{c_{i}c_{j}}(e_{i}-e_{j}),0\leq i<j\leq n\right\},
\end{equation}
\begin{equation}
\label{bn}
B_{n}(c)=\left\{\sqrt{c_i c_j} (e_i\pm e_j),\, 1\le i <j \le n; \quad
\sqrt{2c_i(c_i+c_0)}e_i, \, 1\le i \le n\right\}
\end{equation}
respectively with non-zero parameters $c_0,\dots, c_n$ with non-zero sum.

 In the $A_{n}(c)$ case the corresponding system (\ref{KZ0}) is equivalent to the classical Jordan-Pochhammer system with the solutions, which can be given by the Pochhammer type integrals (see \cite{Ao, OT1} and Section 3 below).
The monodromy of this system is closely related to the classical Gassner representation of the pure braid group \cite{Birman} (see the precise statement and the relation with bending of polygons in \cite{KM}).
For a review of the higher rank representations of the braid group in relation with KZ equation we refer to Kohno \cite{Kohno3, Kohno4}.

In the main part of the paper we study the polynomial solutions of the systems (\ref{KZ0}), which are polynomial $\vee$-parallel vector fields, in relation with the theory of logarithmic vector fields and free arrangements \cite{Saito}.  
Such solutions may exist only for special values of parameter $\kappa,$ which can be shown to be equal to the degree of the corresponding solution.

We call $\vee$-system $\mathcal A$ {\it harmonic} if there are $n=rank \, \mathcal A$ linearly independent polynomial $\vee$-parallel vector fields of degrees $\kappa_1, \dots, \kappa_n$ such that
$$\kappa_1+\dots +\kappa_n=|\mathcal A|$$
is the number of covectors in $\mathcal A$.
We show that for any harmonic $\vee$-system the corresponding vector fields are gradient and freely generate all logarithmic vector fields $Der(\log \Delta)$ as a module over polynomial algebra, which means that the corresponding arrangements are free in Saito's sense \cite{OT}. 
As a corollary by Terao's factorisation theorem \cite{OT} the Poincar\'e polynomial, which is the generating function $P_{\Sigma}(t)=\sum b_i t^i$ of Betti numbers of $\Sigma$, in that case has the form
$$
P_{\Sigma}(t)=\prod_{i=1}^n(1+\kappa_i t).
$$

We conjecture that all the arrangements of $\vee$-systems are free, so the corresponding Poincar\'e polynomials are always factorizable in such a form with integer $\kappa_i$.
We prove this for all known $\vee$-systems \cite{FV, FV2}, by showing that in dimension $n>2$ the corresponding arrangements are equivalent to Coxeter arrangements or their restrictions, which are known to be free \cite{OT2}.

In Section 4 we prove that the classical series of $\vee$-systems (\ref{an}), (\ref{bn}) are harmonic and present the residue formulae for the potentials of the corresponding gradient vector fields (see Theorems 4 and 5). This fact seems to be remarkable since as we show even the restrictions of Coxeter systems in general may not be harmonic.

In the last section we discuss the Coxeter case and the relation of harmonic $\vee$-systems with Saito flat coordinates on the orbit space of Coxeter groups \cite{Saito2, Saito3}. 
We prove that all Coxeter $\vee$-systems are harmonic and find the corresponding potentials. In the case when all the roots are normalised to have the same length these potentials  are known to be precisely the Saito flat coordinates \cite{FS}, so in the non-simply laced cases we have one-parameter deformations of these coordinates, which we describe explicitly.

\section{$\vee$-systems and representations of holonomy Lie algebras}

Let $\Delta$ be a central hyperplane arrangement in $V.$ For any hyperplane $H \in \Delta$ we choose $\alpha_H \in V^*$ such that
$H=\{x \in V: \alpha_H(x)=0\}.$
We will call the corresponding set $$\mathcal A =\{\alpha_H, \, H\in \Delta\}\subset V^*$$ 
an {\it equipment} of $\Delta.$ We will assume that the set $\mathcal A$ generates $V^*.$ Arrangement $\Delta$ is called {\it irreducible} if one cannot decompose $V^*= V_1 \oplus V_2$ such that ${\mathcal A} = ({\mathcal A} \cap V_1) \cup ({\mathcal A}\cap V_2)$.   

Assume now that $V$ is a complex Euclidean space with symmetric non-degenerate bilinear form $G.$ Denote by $\hat \alpha=G^{-1} \alpha$ the vector corresponding to $\alpha \in V^*$ and  look for representations $\rho: \mathfrak g_{\Delta} \rightarrow End(V)$ of holonomy Lie algebra $\mathfrak g_{\Delta}$ of the form
\begin{equation}\label{rep}
\rho(t_\alpha)= \hat \alpha \otimes \alpha, \quad \alpha \in \mathcal A
\end{equation}
for some equipment $\mathcal A$ of $\Delta.$ In general, there are no such equipments, so these representations exist only for special hyperplane arrangements.

To state the theorem we will need the following notion of complex Euclidean $\vee$-system introduced in \cite{FV2}.

Let $\mathcal A$ be a finite set of non-collinear vectors in a complex Euclidean vector space $V\cong V^*.$ We say that the set $\mathcal A$ is {\it well-distributed} in $V$ if the canonical
form (\ref{GEC}) is proportional to the Euclidean form $G.$ The set $\mathcal A$ is called {\it complex Euclidean $\vee$-system}
if it is well-distributed in $V$ and any its two-dimensional subsystem is either reducible (consists of two orthogonal vectors) or well-distributed in the corresponding plane. 

Note that we allow here the canonical form to be zero. If the canonical form (\ref{GEC}) is non-degenerate, then we can use it to define the Euclidean structure on $V$ and we have the definition of the usual $\vee$-system.

\begin{thm}
For any $\vee$-system $\mathcal A$ the formula 
$$
\rho(t_\alpha)= \alpha^\vee \otimes \alpha, \quad \alpha \in \mathcal A
$$
defines a representation of the associated holonomy Lie algebra $\mathfrak g_{\Delta}$. The same is true for complex Euclidean $\vee$-systems and representation (\ref{rep}).

Conversely, if (\ref{rep}) is a representation of the holonomy Lie algebra $\mathfrak g_{\Delta}$ for an irreducible arrangement $\Delta$ with equipment $\mathcal A$  then 
$\mathcal A$ is a complex Euclidean $\vee$-system.
\end{thm}

Thus for given hyperplane arrangement $\Delta$ the description of all representations of holonomy Lie algebra $\mathfrak g_{\Delta}$ of the form (\ref{rep}) is essentially equivalent to the classification of all $\vee$-systems $\mathcal A$ associated to $\Delta.$ 
Note that $\rho$ depends not only on the arrangement, but also on the choice of the equations of the hyperplanes.

Recall first that due to Kohno \cite{Kohno2} the flatness conditions of the universal logarithmic connection (\ref{KZK})
$$
 [\nabla_\xi, \nabla_\eta]=0
$$
are equivalent to the relations (\ref{t_A}).

A similar interpretation of the $\vee$-conditions as flatness of the corresponding $\vee$-connection 
on the tangent bundle $T(\Sigma)\approx \Sigma \times V$ 
\begin{equation}\label{KZ}
\nabla^\vee_\xi=\partial_\xi - \kappa\sum_{\alpha\in {\mathcal A}}\frac{\alpha(\xi)}{\alpha(x)}\alpha^\vee\otimes \alpha, \quad \xi \in V, \, x \in \Sigma
\end{equation}
was pointed out in \cite{V2}. Indeed, it is easy to see that the relation $[\nabla^\vee_\xi, \nabla^\vee_\eta]=0$ is equivalent to the identity
$$
\sum_{\alpha,\beta \in {\mathcal A}}\frac{\alpha\wedge \beta}{\alpha(x)\beta(x)}[\alpha^\vee\otimes \alpha, \beta^\vee\otimes \beta]=0,
$$
which in its turn is equivalent to the commutation relations
\begin{equation}\label{pi}
[\alpha^\vee\otimes \alpha, \sum_{\beta \in {\mathcal A \cap \Pi}}\beta^\vee\otimes \beta]=0
\end{equation}
for all $\alpha \in \mathcal A$ and all 2-dimensional subspaces $\Pi \subset V^*$ containing $\alpha.$
Now if $\Pi$ contains only two covectors $\alpha$ and $\beta$ then we have
$$[\alpha^\vee\otimes \alpha, \beta^\vee\otimes \beta]=\alpha(\beta^\vee) \alpha^\vee\otimes\beta-\beta(\alpha^\vee) \beta^\vee\otimes\alpha$$
which is zero for non-proportional $\alpha$ and $\beta$ only if $$\alpha(\beta^\vee)=\beta(\alpha^\vee)=G_{\mathcal A}(\alpha^\vee, \beta^\vee)=0,$$
which is $\vee$-condition (\ref{vee2}).
If $\Pi$ contains more than two covectors then the commutation relations (\ref{pi}) are equivalent to the property that the restriction of the operator
$\sum_{\beta \in {\mathcal A \cap \Pi}}\beta^\vee\otimes \beta$ on $\Pi$ is proportional to the identity, which coincides with $\vee$-condition (\ref{vee1}).

Similarly, for any complex Euclidean $\vee$-system the corresponding connection
$$
\hat\nabla_\xi=\partial_\xi - \kappa\sum_{\alpha\in {\mathcal A}}\frac{\alpha(\xi)}{\alpha(x)}\hat\alpha\otimes \alpha
$$
is flat, which is equivalent to the relations
\begin{equation}\label{pi*}
[\hat \alpha\otimes \alpha, \sum_{\beta \in {\mathcal A \cap \Pi}}\hat \beta\otimes \beta]=0
\end{equation}
for all $\alpha \in \mathcal A$ and all 2-dimensional subspaces $\Pi \subset V^*$ containing $\alpha$ (cf. \cite{V3}).

To prove the first part of the theorem we note that substitution of (\ref{rep}) into the holonomy Lie algebra relations (\ref{t_A}) gives
the relations (\ref{pi*}).

Now fixing $\alpha$ and summing these relations over all 2-dimensional $\Pi$ containing $\alpha$ we have
\begin{equation}\label{A}
[\hat \alpha\otimes \alpha, \sum_{\beta \in {\mathcal A}}\hat \beta\otimes \beta]=0.
\end{equation}
 Since this is true for all $\alpha \in \mathcal A$, the set $\mathcal A$ generates $V^*$ and the arrangement is irreducible this implies that the operator
 $\sum_{\beta \in {\mathcal A}}\hat \beta\otimes \beta$ is proportional to the identity, or equivalently, that
 $$G_{\mathcal A}=\sum_{\beta \in {\mathcal A}}\beta\otimes \beta=\mu G.$$
 If $\mu\neq 0$ then $G_{\mathcal A}$ is non-degenerate and $\alpha^\vee = \mu^{-1} \hat\alpha$ satisfy $\vee$-conditions (\ref{pi}). If $\mu=0$ then we have complex Euclidean $\vee$-system. This completes the proof.
  
 \section{$\vee$-systems and gradient logarithmic vector fields}
 
 One of the main problems in the theory of $\vee$-systems is the characterisation 
 of the corresponding hyperplane arrangements, see e.g. \cite{SchV}.
 Since in dimension 2 any covector system is a $\vee$-system, 
 the problem starts from dimension 3. 
 
We would like to link this problem  with the characterization of free arrangements in the theory of logarithmic vector fields initiated by Kyoji Saito \cite{Saito}. We start with a brief review of this theory, 
 mainly following Orlik and Terao \cite{OT}.
 
Consider a hyperplane arrangement $\Delta \subset \mathbb C^n$. A polynomial vector field $X=\xi_i(z)\frac{\partial}{\partial z_i}$ on $\mathbb C^n$ is called {\it logarithmic}  if it is tangent to every hyperplane $H \in \Delta$.
The hyperplane arrangement $\Delta$ is {\it free} if the space of all logarithmic vector fields $Der(\log \Delta)$ is free as the module over polynomial algebra $P_n=\mathbb C[z_1, \dots, z_n]$ (see \cite{OT,Saito}).
The degrees $b_1, \dots, b_n$ of the corresponding homogeneous generators $X_1,\dots, X_n$ are called the {\it exponents} of the arrangement:
$$\exp \, \Delta=\{b_1, \dots, b_n\}.$$ Here the degree of a homogeneous polynomial vector field $X=\xi_i(z)\frac{\partial}{\partial z_i}$
is defined as the degree of any of its non-zero components: $\deg \, X=\deg\, \xi_i.$

Saito's criterion \cite{OT} says that $\Delta$ is free if and only if there are $n$ homogeneous linearly independent over $P_n$ logarithmic vector fields $X_1,\dots, X_n$ such that the sum of the degrees equals the number of hyperplanes $|\Delta|$:
\begin{equation}\label{saito}
\sum_{i=1}^n \deg \, X_i = |\Delta|.
\end{equation}
Such fields can be chosen as free generators of the module $Der(\log \Delta).$

However, a satisfactory characterization of all free arrangements is still an open problem. There is a conjecture due to Terao, that the freeness property is combinatorial (see \cite{OT}, page 154), but there is still not much evidence in its favour.

Probably the most remarkable result in this area is the following {\it Factorization Theorem} proved by Terao \cite{Terao}:
Poincar\'e polynomial of the complement $\Sigma=\mathbb C^n \setminus \Delta$ for a free arrangement $\Delta$ has the form
\begin{equation}\label{terao}
P_{\Sigma}(t)=\prod_{i=1}^n(1+b_it)
\end{equation}
with $b_1,\dots, b_n$ being the exponents of $\Delta.$
This is a far-going generalisation of Arnold's formula
$$
P_{\Sigma_{n+1}}(t)=(1+t)(1+2t)\dots(1+nt)
$$
for the Poincar\'e polynomial of the configuration space of $n+1$ distinct points on the plane, corresponding to $A_n$-type arrangement, 
see \cite{Arnold}.

It is well-known (Arnold, Saito) that all Coxeter arrangements are free with the exponents $b_i=m_i$ being the exponents of the corresponding Coxeter (finite reflection) group $G$  \cite{Humphreys}.
The corresponding generators $X_i =grad\, f_i, \, i=1,\dots, n,$ where $f_1,\dots, f_n$ are some free generators of the corresponding algebra of polynomial $G$-invariants $\mathbb C[z_1, \dots, z_n]^G,$
which exist by Chevalley theorem. Indeed, it is easy to see that the corresponding fields are logarithmic and, by Saito's criterion, generate $Der(\log \Delta)$ because the sum of the exponents of a Coxeter group is known to be the number of the reflection hyperplanes:
$$
m_1+\dots+m_n=|\Delta|,
$$
see e.g. \cite{Humphreys}.

It is known also that any linear arrangement in $\mathbb C^2$ is free and that
a generic arrangement in $\mathbb C^n$ with $n>2$ is not free \cite{OT}.

The arrangement $\Delta$ is called {\it hereditarily free} if it is free and all restriction arrangements to the hyperplanes of $\Delta$ and their intersections are also free  \cite{OT}. The property of $\Delta$ being free is not hereditary \cite{OT}, but it is known that all Coxeter arrangements 
are hereditarily free \cite{OT2}.
 
{\bf Conjecture.} {\it For any $\vee$-system $\mathcal A$ the associated arrangement $\Delta_{\mathcal A}$ is hereditarily free.}
 
We have shown that the class of $\vee$-systems is closed under the restriction \cite{FV2}, so to prove the conjecture it is enough to show that  
$\Delta_{\mathcal A}$ is free. 
In particular, this would imply by Terao's theorem that the corresponding Poincar\'e polynomial $P_{\Sigma_\mathcal A}(t)$ is factorizable in the form
(\ref{terao}), which would be already a strong topological restriction of the arrangement.


 \begin{thm}
For all known $\vee$-systems \cite{FV, FV2} the conjecture is true.
\end{thm}

\begin{proof} In dimension 2 this follows from the fact that in that case any system is a $\vee$-system and any line arrangement is free \cite{OT}. 

For the classical series of $\vee$-systems $A_n(c)$ the corresponding arrangements (\ref{an}) coincide with the Coxeter arrangement of type $A_n,$ and thus are free with the exponents $1,2,\dots, n.$

In the $B_n(c)$ case the corresponding arrangements have Coxeter type $B_n$ unless $c_0+c_i=0$ for some $i$. In that case we have Zaslavsky arrangements $D_n^k$ consisting of hyperplanes $x_i=0, \, i=1, \dots, k, \, 1 \leq k \leq n$ and $x_i\pm x_j=0, \, 1 \leq i < j \leq n.$ These arrangements appear also as some restrictions of the systems of type $D_N$, see \cite{OT, Zaslavsky}, and thus are free with the exponents $1,3,5,\dots, 2n-3, k+n-1$ \cite{OS}.

In the $F_4$ case we have the $\vee$-system consisting of the covectors 
$$e_i\pm e_j, \, 2t e_i, \, t(e_1\pm e_2 \pm e_3 \pm e_4).$$
If $t\neq 0$ this gives the Coxeter arrangement of type $F_4$ with the exponents $1,5,7,11.$
If $t=0$ we have the Coxeter $D_4$ arrangement with the exponents $1,3,3,5.$

For the Coxeter restrictions we will use the notations from \cite{FV} as pairs $(G,H),$ where $G$ is Coxeter group and $H$ is the corresponding parabolic subgroup (see also \cite{OS}).

There are two restrictions of $F_4$, but they are equivalent (see \cite{FV}) to
$$
F_3(t)=\{e_1 \pm e_2, e_2 \pm e_3, e_1 \pm e_3, \sqrt{a}\,e_1, \sqrt{a}\,e_2,
\sqrt{a}\,e_3, t\sqrt{2}\,(e_1\pm e_2 \pm e_3)\}, \,\, a=4t^2+2,
$$
where $t^2 \neq -1/2.$ If $t\neq 0$ we have the arrangement of 13 planes equivalent to the Coxeter restriction $(E_8, D_5)$, which is free with the exponents $1,5,7$ \cite{OS}.  If $t=0$ we have the $B_3$ arrangement with the exponents $1,3,5.$

Consider now the generalised systems related to exceptional basic classical Lie superalgebras and their deformations \cite{SV, FV2}.

In the $AB_4$ case related to Lie superalgebra of type $AB(1,3)$ we have generally 18 covectors 
$$\sqrt{a} e_i, \, i=1,2,3, \,\, \sqrt{b}(e_i\pm e_j), \, 1\leq i < j \leq 3,\,\,  \sqrt{c} e_4, \,\, \frac{1}{2}(e_1\pm e_2 \pm e_3 \pm e_4),$$
where 
$$
a=\frac{3k+1}{2}, \,\, b=\frac{3k-1}{4}, \,\, c=\frac{1-k}{2k},
$$
$k\neq 0, -1/3$ is an arbitrary parameter (see \cite{SV}). When $k=-1/3$ the canonical form is zero, so we have complex Euclidean $\vee$-system but not a $\vee$-system.

If all $a,b,c$ are not zero then we have the free arrangement equivalent to the Coxeter restriction $(E_7, A_3)$ of $E_7$-type system, which is free with the exponents  $1,5,5,7$ (see \cite{OS}).
If $c=0, \, k=1$ we have the hyperplane $x_4=0$ disappeared, so we can apply Terao's Addition-Deletion theorem (see Theorems 4.49 and 4.51 in \cite{OT}) to claim that the corresponding arrangement of 17 hyperplanes is free with the exponents $1,4,5,7.$
Alternatively, one can check that the corresponding $\vee$-system is just the Coxeter restriction $(E_6, A_1\times A_1)$ (see \cite{FV}) and use the results from \cite{OS,OT2}.


Finally, if $b=0$, $k=1/3$, we have the $\vee$-system consisting of 12 covectors
$$
e_i, i=1,2,3,4, \,\, \,\, \frac{1}{2}(e_1\pm e_2 \pm e_3 \pm e_4),
$$
 which is equivalent to Coxeter system of type $D_4$ and thus is free with exponents $1,3,3,5$.
 
 We should consider also the 3D restrictions. There are two types of these restrictions consisting of
the following covectors \cite{FV2}
$$
(AB_4(t),A_1)_1= \{  \sqrt{2(2t^2+1)} e_1, \, 2 \sqrt{2(t^2+1)} e_2, t
\sqrt{\frac{2(2t^2-1)}{t^2+1}}e_3, $$ $$ \sqrt{2}(e_1\pm e_2), \,
t\sqrt{2}(e_1 \pm e_3), \, t(e_1 \pm 2 e_2 \pm e_3)\}
$$
with $t^2\neq -1, -1/2$, and 
$$
(AB_4(t),A_1)_2 = \{ e_1+e_2, e_1+e_3, e_2+e_3, \sqrt{2}e_1,
\sqrt{2}e_2, \sqrt{2}e_3,
\frac{t\sqrt{2}}{\sqrt{t^2+1}}(e_1+e_2+e_3),
$$
$$
\frac1{\sqrt{4t^2+1}}(e_1-e_2), \frac1{\sqrt{4t^2+1}}(e_1-e_3),
\frac1{\sqrt{4t^2+1}}(e_2-e_3)\}
$$
with $t^2\neq -1, -1/2, -1/4.$
If $t^2 \neq 0, 1/2$ then the 3D arrangement $(AB_4(t),A_1)_1$ consists of 11 planes, is equivalent to the Coxeter restriction $(E_7, A_1\times A_3)_2,$ and thus is free with the exponents $1,5,5$ \cite{OS}.
When $t=0$ we have the reducible arrangement of 5 planes $B_2\times A_1$, which is free with exponents $1,1,3$
(see Prop. 4.28 in \cite{OT}). When $t^2=1/2$ we have arrangement of 10 planes equivalent to Coxeter restriction  $(E_6, A_1^3)$ which is free with exponents $1,4,5$ \cite{OS}.
Similarly, if $t\neq 0$ the 3D arrangement $(AB_4(t),A_1)_2$ consists of 10 planes and is equivalent to the Coxeter restriction $(E_6, A_1 \times A_2)$, which is free with the exponents $1,4,5$ \cite{OS}.
When $t=0$ we have 9 planes, forming $B_3$ arrangement with the exponents $1,3,5.$
 
The $\vee$-systems of type $G_3$ related to the Lie superalgebra $G(2,1)$ consist of covectors \cite{SV} 
$$G_3(t)=\{\sqrt{a}e_1, \sqrt{a}e_2, \sqrt{a}(e_1+e_2),
\sqrt{b}(e_1-e_2), \sqrt{b}(2 e_1+e_2), 
\sqrt{b}(e_1+ 2 e_2), \,\, \sqrt{c} e_3,$$
$$
e_1 \pm e_3,\,\,  e_2 \pm e_3, \,\, e_1+e_2 \pm e_3\}, \,\,\, a=2t+1, \, b=\frac{2t-1}{3}, \, c= \frac{3}{t},$$
depending on the parameter $t\ne 0, -1/2.$ In the case $t\neq 1/2$ we have the arrangement of 13 planes equivalent to the Coxeter restriction $(E_7, A_2^2)$ (or, equivalently, $(E_8,A_5)$), which is free with the exponents $1,5,7$ (see \cite{FV, OS}). When $t=1/2$ we have the arrangement of 10 planes, which can be shown to be equivalent to the Coxeter restriction $(E_6, A_1^3)$, and thus is free with the exponents $1,4,5$ (see \cite{OS}).



Finally the $\vee$-systems of type $D_3$ related to the exceptional Lie superalgebra $D(2,1,\lambda)$ 
consist of covectors \cite{FV2}
$$
D_3(t,s)=\{e_1 \pm e_2 \pm e_3, \, \sqrt{2(s+t-1)}e_1, \, \sqrt{\frac{2(s-t+1)}{t}}e_2, \, \sqrt{\frac{2(t-s+1)}{s}}e_3,\}
$$
where $s,t$ are non-zero parameters, such that $s+t+1\neq0.$
For generic values of parameters the corresponding arrangement is equivalent to the Zaslavsky configuration $D_3^1=(D_6, A_3)$ with the exponents $1,3,3.$ If one of the coefficients of $e_1, e_2, e_3$ vanishes, then we have type $A_3$ arrangement with the exponents $1,2,3.$
This analysis together with all other Coxeter restrictions exhausts all the cases and completes the proof.
\end{proof}

{\bf Remark.} {\it Note that for the complex Euclidean $\vee$-systems the conjecture is not true. A counterexample is given by the $\vee$-system of type $F_3(t)$ with $t^2=-\frac{1}{2}$, consisting of the following 10 vectors in $\mathbb C^3$
$$\mathcal A=\{e_1 \pm e_2,\, e_1 \pm e_3,\, e_2 \pm e_3, \, i(e_1\pm e_2 \pm e_3)\}.$$
The corresponding Poincar\'e polynomial
$$P_{\mathcal A}(t)=(1+t)(1+9t+26t^2)$$
is not factorizable, so the arrangement is not free and is not combinatorially equivalent to any Coxeter restriction. Note that the corresponding canonical form $G_{\mathcal A}=0$ in this case.}

We are going to show now that at least for a subclass of $\vee$-systems we can find the corresponding generating logarithmic vector fields $X_1, \dots, X_n$ as polynomial $\vee$-parallel vector fields (\ref{KZ0}) for special values of $\kappa$ being the exponents $b_1,\dots, b_n.$
 
 \begin{thm}\label{gradlog}
The polynomial solutions $\psi$ of (\ref{KZ0}) are gradient logarithmic vector fields for the corresponding arrangement $\Delta_{\mathcal A}$ with the degrees
\begin{equation}\label{deg}
deg \, \psi=\kappa.
\end{equation}
\end{thm}

\begin{proof}
To prove this it is convenient to choose an orthonormal basis in $V$, so that the canonical form $G_{\mathcal A}$ becomes standard. Then we can identify $V$ and $V^*$ with $\mathbb C^n,$ so that $\alpha^\vee=\alpha$ and
 $$
\sum_{\alpha\in {\mathcal A}}\alpha_i\alpha_j=\delta_{ij}, \quad i, j=1, \dots, n, 
 $$
 where $\alpha_i$ is the $i$-th coordinate of $\alpha$.
The system (\ref{KZ0}) takes the form
 \begin{equation}\label{sections}
\partial_i\psi_j = \kappa\sum_{\alpha\in {\mathcal A}}\frac{\alpha_i \alpha_j}{(\alpha,z)}(\alpha,\psi), \quad  z, \psi(z) \in \mathbb C^n, \,  i, j=1, \dots, n.
\end{equation}

Now from (\ref{sections}) it is immediate that $\partial_i\psi_j=\partial_j\psi_i$, so 
\begin{equation}\label{gr1}
\psi_i=\partial_i F,\,\, \, \quad i=1,\dots, n
\end{equation} 
for some polynomial potential $F(z).$ The fact that $\psi$ is logarithmic follows from the regularity of the left hand side on the hyperplane $(\alpha,z)=0,$ which implies
that $(\alpha,\psi)=0$ on this hyperplane, so that $\psi$ is tangent.
To find the degree of $\psi$ multiply  the relations (\ref{sections}) by $z_i\beta_j$ and add over all $i,j$ to have
$$
E (\beta, \psi) = \kappa \sum_{\alpha\in {\mathcal A}} (\alpha, \beta)(\alpha, \psi) = \kappa (\beta, \psi),
$$
where $E= \sum_{i=1}^n z_i \partial_i$ is the Euler vector field.
\end{proof}

 The potential $F$ of a $\vee$-parallel vector field $\psi$ can be defined in coordinate-free way by the relation
\begin{equation}
\label{pot}
\a (\psi) = G_{\mathcal A}(\a^\vee, \psi) = d F (\a^\vee) = \p_{\a^\vee} F
\end{equation}
for any $\a \in V^*$. The parallel transport condition
\beq{fl1}
\p_\xi \psi = \kappa \sum_{\a \in {\mathcal A}} \frac{\a(\xi) \a(\psi)}{\a(z)} \a^\vee = \kappa \sum_{\a \in {\mathcal A}} \frac{\a(\xi) \p_{\a^\vee}F}{\a(z)} \a^\vee 
\eeq
implies that the potential $F$ satisfies compatible system of the Euler-Poisson-Darboux type equations
\beq{EPD}
\p_\xi\p_\eta F = \kappa \sum_{\a \in {\mathcal A}} \frac{\a(\xi) \a(\eta)}{\a(z)} \p_{\a^\vee}F, \quad \xi,\eta \in V.
\eeq


So the question is for which integer values of parameter $\kappa$ do the polynomial solutions of (\ref{KZ0}) exist, and whether we can find enough such solutions to generate all logarithmic vector fields over polynomial algebra.
Note that for $\kappa=1$ we always have the solution $\psi_i=z_i, \, i=1, \dots, n$ corresponding to the Euler vector field $\psi=E$.

To understand the situation better let us consider the case of rank 2 systems $\mathcal A.$ In this case the gradient generators of logarithmic vector fields may not exist. It is well-known \cite{OT} that any such arrangement is free and
 $Der(\log \Delta_\mathcal A)$ is generated by Euler vector field $E$ and $$X=(\partial_2 Q) \partial_1 - (\partial_1 Q) \partial_2,$$
where $Q=\prod _{\alpha \in \mathcal A}(\alpha, z).$ The last vector field is gradient if and only if $Q$ is harmonic:
$$\Delta Q=0, \quad \Delta= \partial_1^2+\partial_2^2$$ which in general is not the case.
Indeed, consider a particular case of 4 lines with 
$$
Q=x_1x_2(x_1-x_2)(x_1-ax_2).
$$
Then $\Delta Q=2(1+a)(-x_1^2+3x_1x_2-x_2^2)$, which vanishes only when $a=-1$, so the lines form a harmonic set (which means that their cross-ratio is $-1$), projectively equivalent to $B_2$ case. Adding to $X$ a multiple of Euler field $E$ also would not make it gradient. Indeed if $X+f(x)E=grad F$ then
$$\Delta Q=x_2\partial_1 f-x_1\partial_2 f =\partial_\xi f,$$
where $\xi=(x_2,-x_1).$ Since vector field $\xi$ has closed circular orbits, the necessary condition for the existence of polynomial $f$ is
$\int_\gamma \Delta Q dt=0,$ where $\gamma$ is the circle $x_1= \cos t,\, x_2 = \sin t.$ In our case
$$\int_\gamma \Delta Q dt=2(1+a)\int_0^{2\pi} (3\cos t \sin t -1) dt=-4(1+a)\pi,$$ which is zero only if $a=-1.$
One can also check that  in general the corresponding systems (\ref{sections}) do not have polynomial solutions for $\kappa=3$ if $a\neq-1$. 

This means that we are dealing with a special subclass of both free arrangements and $\vee$-systems.
This motivates the following definition.

We say that $\vee$-system $\mathcal A$ of rank $n$ is {\it harmonic} if the corresponding system (\ref{KZ0}) has $n$ linearly independent (at generic point) polynomial solutions for $\kappa=\kappa_1, \dots, \kappa_n$ such that
\begin{equation}\label{harmon}
\kappa_1+\dots+\kappa_n=|\mathcal A|,
\end{equation} 
where $|\mathcal A|$ is the number of covectors in $\mathcal A$.

\begin{thm}
The arrangement $\Delta$ of any harmonic $\vee$-system is free with exponents $b_i=\kappa_i, \, i=1,\dots, n$ and the Poincar\'e polynomial of $ \Sigma = V \setminus \Delta$ has the form
\begin{equation}\label{poinc}
P_{\Sigma}(t)=\prod_{i=1}^n(1+\kappa_i t).
\end{equation}
\end{thm}

The proof follows immediately from Theorem \ref{gradlog} and the Saito criterion.

As one can see from the Euler-Poisson-Darboux type equations \eqref{EPD}  the corresponding potentials $F_1,\dots, F_n$ belong to the algebra of {\it quasi-invariants} of $\mathcal A$
\begin{equation}\label{quasi}
\mathcal Q_{\mathcal A}=\{ p(z) \in \mathbb C[z_1,\dots, z_n]: \partial_{\alpha^\vee}|_{\alpha(z)=0} p(z)=0,\quad \alpha \in \mathcal A \}.
\end{equation}
It would be interesting to understand their role for these algebras (cf. \cite{FVimrn}). 

As we will see now for the classical series the corresponding potentials turn out to be certain deformations of Saito's generators of the algebra of invariants.

 
\section{Analysis of the classical series}

Consider first $\vee$-systems of type $A_{n}$ from \cite{CV}:
$$
A_{n}(c)=\left\{\sqrt{c_{i}c_{j}}(e_{i}-e_{j}),0\leq i<j\leq n\right\},
$$
where all $c_i$ are assumed to be non-zero.
One can check that the corresponding canonical form is non-degenerate if 
$$
\sigma=c_0+c_1+\dots + c_n \neq 0,
$$ 
and the vector $\alpha^\vee$ for $\alpha=\sqrt{c_ic_j} (e_i-e_j)$ has the form
$$\alpha^\vee=\sigma^{-1} \sqrt{c_i c_j}(c_i^{-1}e_i-c_j^{-1} e_j)$$
(see \cite{CV}).
The corresponding KZ equations $\nabla^\vee_\xi \psi=0$ with $\nabla^\vee_\xi$ given by (\ref{KZ}) and $\psi=(\psi_0,\ldots, \psi_n)\in V^*$ have the form 
\begin{equation}\label{ankz}
\partial_i \psi_j=-\kappa\sigma^{-1}\frac{c_j \psi_i-c_i \psi_j}{x_i-x_j}, \quad i\neq j
\end{equation}
with $\partial_i \psi_i$ determined from the relation $\psi_0+\dots+\psi_n=0:$ 
\begin{equation}\label{ankz2}
\partial_i \psi_i=\kappa\sigma^{-1}\sum_{j \neq i} \frac{c_j \psi_i-c_i \psi_j}{x_i-x_j}.
\end{equation}

These equations are nothing but the Jordan-Pochhammer linear system 
for the integrals of the hypergeometric type 
$$I_\lambda(x_0,\dots, x_n)=\int_\gamma \prod_{j=0}^n(x-x_j)^{\lambda_j} dx$$ 
(see e.g. Aomoto \cite{Ao}, Orlik and Terao \cite{OT1}, formula (1) on page 71, Pavlov \cite{Pavlov}, Couwenberg, Heckman,  Looijenga   \cite{CHL}, Section 2.3, Looijenga   \cite{L}). More precisely, we have the following

\begin{thm}\label{potentialsA}
The $\vee$-systems $A_n(c)$ are harmonic 
with the potentials given by 
the Pochhammer type integrals
\begin{equation}\label{integ}
F_\kappa(x_0,\dots, x_n)=\frac{1}{2\pi i}\int_\gamma \prod_{j=0}^n(x-x_j)^{\lambda_j} dx, \quad \kappa=1,2, \dots, n
\end{equation}
where 
$\lambda_j=\kappa \frac{c_j}{\sigma}$ and contour $\gamma$ is a large circle surrounding all $x_0,\dots, x_n.$
\end{thm}

\begin{proof}
Let $\Phi_{\lambda}(x; x_0,\dots, x_n)= \prod_{j=0}^n(x-x_j)^{\lambda_j}$. Then $I_\lambda(x_0,\dots, x_n)=\int_\gamma \Phi_{\lambda}(x; x_0,\dots, x_n) dx$ and 
\begin{equation}\label{ankzs}
\psi_j=\partial_j I_\lambda(x_0,\dots, x_n)=-\lambda_j\int_\gamma \frac{\Phi_{\lambda}(x; x_0,\dots, x_n)}{x-x_j}dx.
\end{equation}
Note that for the chosen contour $\gamma$ the integral (\ref{integ}) is well-defined if and only if 
$\lambda_0+\lambda_1+\dots +\lambda_n$ is an integer, which we will assume to be the case.
Then we have
$$\psi_0+\dots+\psi_n= - \int_\gamma d\Phi_\lambda=0.$$
Consider the derivative
$$
\partial_i \psi_j=\lambda_i\lambda_j\int_\gamma \frac{\Phi_{\lambda}}{(x-x_i)(x-x_j)}dx=\frac{\lambda_i\lambda_j}{x_i-x_j}\int_\gamma (\frac{\Phi_{\lambda}}{x-x_i}-\frac{\Phi_{\lambda}}{x-x_j})dx=\frac{\lambda_i\psi_j-\lambda_j \psi_i}{x_i-x_j}
$$
if $i \neq j$ and
$$
\partial_i \psi_i=-\sum_{j \neq i}\frac{\lambda_i\psi_j-\lambda_j \psi_i}{x_i-x_j},
$$
which coincides with the equations (\ref{ankz}), (\ref{ankz2}) with $\lambda_j=\kappa \frac{c_j}{\sigma}.$ 

Note that since $\lambda_0+\dots + \lambda_n=\kappa,$ so we need $\kappa$ to be integer.
We claim that if we choose simply the smallest $\kappa=1,2,\dots,n$ then we will have 
the basic gradient logarithmic vector fields $X$ with components
$$
\xi_i=\lambda_i^{-1} \psi_i, \quad i=0,\dots, n
$$
(note that the canonical form is not standard in this case).
Indeed, $\Phi_{\lambda}$ is meromorphic in $x$ at infinity with the expansion
\begin{equation}
\label{phiexpand}
\Phi_{\lambda}=x^{\kappa} \prod_{i=0}^n(1-\frac{x_i}{x})^{\lambda_i}=x^{\kappa} \prod_{i=0}^n(1-\lambda_i\frac{x_i}{x}+\frac{\lambda_i(\lambda_i-1)}{2}\frac{x_i^2}{x^2}+\dots).
\end{equation}
The contour integral (\ref{integ}) is simply the coefficient at $x^{-1}$ in this expansion, so $F_\kappa$ is polynomial both in $x_0,\dots, x_n$ and $\lambda_0,\dots, \lambda_n.$
Simple algebraic arguments show that the determinant of the matrix of the partial derivatives
$||\partial F_i/\partial x_j||, i,j=0,\dots,n$ is equal to
 $$\det ||\partial F_i/\partial x_j||=\lambda_0\dots \lambda_n \prod_{i<j} (x_i-x_j).$$
 This implies their independence for all non-zero $\lambda_i$, or equivalently, for all non-zero $c_i$ with $\sum c_i=\sigma \neq 0.$ 
Since the degrees of these polynomials in $x$ are the same as in the non-deformed case: $2,3, \dots, n+1,$ by Saito's criterion we have the claim.  \end{proof}

There is a more explicit way to represent potentials $F_{\kappa}$ from Theorem \ref{potentialsA}. Let us introduce deformed Newton sums by
$$
p_s^\lambda = \sum_{i=0}^{n} \lambda_i x_i^s, \quad s \in \mathbb N.
$$
 Further, for a partition $\mu$ define the corresponding deformed power sum  by
$$
p_\mu^\lambda = (p_1^\lambda)^{m_1} (p_2^\lambda)^{m_2}\ldots,
$$
where $m_i$ is the number of parts $i$ in $\mu$. Let $l(\mu)$ be the number of non-zero parts in the partition $\mu$ and define
$$z_\mu = \prod_{j \ge 1} (j^{m_j} m_j!).$$

The $\vee$-system $A_n(c)$ is considered in the subspace where  $p_1^\lambda \equiv 0,$ but it is convenient to keep this polynomial in the formulas below.

\begin{thm}
Potentials \eqref{integ} have the form
\begin{multline}
\label{detfa}
F_\kappa= \frac{(-1)^{\kappa+1}}{(\kappa+1)!} \det  
\begin{pmatrix}
    p_1^\lambda & 1 &0& 0  \ldots  & 0 \\
    p_2^\lambda & p_1^\lambda & 2 & 0 \ldots  & 0 \\
    \vdots & \vdots & \vdots & \ddots & \vdots \\
   p_{\kappa}^\lambda & p_{\kappa-1}^\lambda  & p_{\kappa-2}^\lambda & \dots  & \kappa \\ 
   p_{\kappa+1}^\lambda & p_\kappa^\lambda  & p_{\kappa-1}^\lambda  & \dots  & p_1^\lambda 
\end{pmatrix}
= \sum_{\mu: |\mu|=\kappa+1} (-1)^{l(\mu)} z_\mu^{-1} p_{\mu}^\lambda,
\end{multline}
where $\kappa = 1,\ldots, n$ and the sum is taken over all partitions $\mu$ of $\kappa+1$.
\end{thm}
\begin{proof}
We know that $F_\kappa(x_0,\dots, x_n)$ equals  the coefficient at $x^{-\kappa-1}$ in the expansion of the function $\widehat\Phi_\lambda = x^{-\kappa} \Phi_\lambda$, where $\Phi_\lambda$ is given by \eqref{phiexpand}.
On the other hand 
$$
\widehat \Phi_\lambda = e^{\log  \prod_{i=0}^n(1-\frac{x_i}{x})^{\lambda_i}} = e^{\sum_{i=0}^n \lambda_i \log  (1-\frac{x_i}{x})} 
= e^{-p_1^\lambda x^{-1} -\frac12 p_2^\lambda x^{-2} -\frac13 p_3^\lambda x^{-3} -\ldots}.
$$
The required coefficient is a polynomial in $p_j^\lambda$ with $1\le j \le \kappa+1$. Note that all these polynomials $p_j^\lambda$ are algebraically independent. Therefore the required coefficient is a polynomial in variables $p_j^\lambda$ whose coefficients do not depend on $\lambda$, and it is sufficient to establish the statement when $\lambda_i=1$ for all $i,$  which is well known (see pages 25, 28 in \cite{Macd}). (Note also that the statement for $\lambda_0=\lambda_1=\ldots=\lambda_n$ is contained in \cite{M}, and that a  related analysis of functions $F_\kappa$ is contained in the recent paper \cite{KK}, Remark 6.6).
\end{proof}


Consider now the $\vee$-systems of $B_n$-type \cite{CV}
$$
B_{n}(c)=\left\{\sqrt{c_i c_j} (e_i\pm e_j),\, 1\le i <j \le n; \quad
\sqrt{2c_i(c_i+c_0)}e_i, \, 1\le i \le n\right\},
$$
where all $c_i$ with $i\ge 1$ are assumed to be non-zero.
Let us also assume for the beginning that $c_i+c_0 \neq 0$ for all $i=1,\dots,n,$ so the corresponding 
arrangement is of type $B_n.$

The canonical form has the matrix $G=2\sigma C, \, C=diag \,(c_1,\dots, c_n)$ with 
$$
\sigma=c_0+c_1+\dots + c_n,
$$
so for $\alpha=\sqrt{c_i c_j} (e_i\pm e_j)$ we have
$$\alpha^\vee=2^{-1}\sigma^{-1} \sqrt{c_i c_j}(c_i^{-1}e_i\pm c_j^{-1}e_j)$$
and for $\alpha=\sqrt{2c_i(c_i+c_0)}e_i$ we have
$$\alpha^\vee=(2\sigma c_i)^{-1}\sqrt{2c_i(c_i+c_0)}e_i.$$
The corresponding equations (\ref{KZ}) for $\psi=(\psi_1, \ldots, \psi_n)\in V^*$ have the form
\begin{equation}\label{bnkz}
2\sigma\kappa^{-1}\partial_i \psi_j=-\frac{c_j \psi_i-c_i \psi_j}{x_i-x_j}+\frac{c_j \psi_i+c_i \psi_j}{x_i+x_j}, \quad i\neq j,
\end{equation}
\begin{equation}\label{bnkz2}
2\sigma\kappa^{-1}\partial_i \psi_i=\sum_{j\neq i}(\frac{c_j \psi_i-c_i \psi_j}{x_i-x_j}+\frac{c_j \psi_i+c_i \psi_j}{x_i+x_j}) +\frac{2(c_i+c_0)\psi_i}{x_i}.
\end{equation}
Consider the product
$$\Phi_{\lambda}= \prod_{j=1}^n(x^2-x_j^2)^{\lambda_j}x^{2\lambda_0}$$
and the corresponding integral
\begin{equation}\label{integB}
J_\lambda(x_1,\dots, x_n)=\int_\gamma \Phi_{\lambda} dx=\int_\gamma \prod_{j=1}^n(x^2-x_j^2)^{\lambda_j}x^{2\lambda_0} dx,
\end{equation}
where $\gamma$ as before is a large circle. The integral is well-defined if the sum
$$2(\lambda_0+\lambda_1+\dots +\lambda_n) \in \mathbb Z.$$

\begin{thm}
\label{bnth}
$\vee$-systems $B_n(c)$ with $c_j+c_0\neq 0$ for all $j=1,\dots, n$ are harmonic with the corresponding potentials $F_k=\frac{1}{2\pi i}J_{\lambda},$ where $J_{\lambda}$ are contour integrals (\ref{integB}) with
$\lambda_i=(2k-1) \frac{c_i}{2\sigma}$ and $k=1,\dots, n.$ The corresponding value of $\kappa$ is
$2k-1$.
\end{thm}

\begin{proof}
We have
\begin{equation}\label{bnkzs}
\psi_j=\partial_jJ_\lambda(x_1,\dots, x_n)=-2\lambda_j\int_\gamma \frac{x_j \Phi_{\lambda}}{x^2-x_j^2}dx.
\end{equation}
One can easily check that
$$
\partial_i \psi_j=4\lambda_i\lambda_j\int_\gamma \frac{x_ix_j \Phi_{\lambda}}{(x^2-x_i^2)(x^2-x_j^2)}dx=-\frac{\lambda_j\psi_i-\lambda_i \psi_j}{x_i-x_j}+\frac{\lambda_j\psi_i+v_i \psi_j}{x_i+x_j}
$$
when $i \neq j.$ When $i=j$ we have
$$
\partial_i \psi_i=-2\int_\gamma (\frac{\lambda_i}{(x^2-x_i^2)}-\frac{2\lambda_i(\lambda_i-1)x_i^2}{(x^2-x_i^2)^2}) \Phi_{\lambda} dx =-2\int_\gamma (\frac{\lambda_i(2\lambda_i-1)}{(x^2-x_i^2)}-\frac{2\lambda_i(\lambda_i-1)x^2}{(x^2-x_i^2)^2})\Phi_{\lambda} dx.
$$
On the other hand 
$$
\sum_{j\neq i}(\frac{\lambda_j \psi_i-\lambda_i \psi_j}{x_i-x_j}+\frac{\lambda_j \psi_i+\lambda_i \psi_j}{x_i+x_j}) +\frac{2(\lambda_i+\lambda_0)\psi_i}{x_i}$$
$$=-2\int_\gamma (\sum_{j\neq i}\frac{2\lambda_i\lambda_jx^2 \Phi_{\lambda}}{(x^2-x_i^2)(x^2-x_j^2)}+\frac{2\lambda_i(\lambda_i+\lambda_0)\Phi_{\lambda} }{x^2-x_i^2})dx.
$$
Since the difference of the right hand sides of the last two formulas is the integral of the total derivative $\int_\gamma d\frac{x \Phi_{\lambda}}{x^2-x_i^2}$, 
we see that the integrals (\ref{bnkzs}) satisfy equations (\ref{bnkz}),(\ref{bnkz2}) with $\lambda_j=\kappa\frac{c_j}{2\sigma}.$

Note that $2(\lambda_0+\lambda_1+\dots +\lambda_n)=\kappa$, so $\kappa$ must be an integer.
It is easy to see that the integral (\ref{integB}) vanishes for even $\kappa,$ so the minimal values of $\kappa$ are $1,3,5,\dots, 2n-1,$ which are the exponents of the Weyl group $B_n$.
Similarly to the $A_n$ case, one can show that the corresponding logarithmic vectors fields are independent for all non-zero $c_i$, so the claim follows again from Saito's criterion.
\end{proof}

As in the $A_n$ case, the potential $F_k$ can be  given by formula \eqref{detfa} with $\kappa = k-1$, where one has to replace $p_s^\lambda$ with 
$$
q_{s}^\lambda =\sum_{i=1}^n \lambda_i x_i^{2s}.
$$ 


If $c_1=c_2=\dots=c_k=-c_0$ for some $k=1,\dots, n-1$ then associated arrangement $\Delta$ is not of Coxeter type.
It was studied first by Zaslavsky \cite{Zaslavsky} and is usually denoted as $D_n^k$ \cite{OS}. It is known to be free with the exponents of $1, 3, \dots, 2n-3, 2n-k-1$ (see \cite{OS,OT}).
The first $n-1$ generating potentials $F$ can be found by the same integrals (\ref{integB}) with $\lambda_i=\kappa \frac{c_i}{2\sigma}$ and $\kappa=1,3, \dots, 2n-3,$ but the last one of the required degree $2n-k$ appears not to exist for general values of the remaining parameters $c_{k+1}, \dots, c_n$ (see below the example with $k=2, n=3, c_3=3$). 

For special $c$ this is however possible. Let 
$$c_1=c_2=\dots=c_k=-c_0=1, \, c_{k+1}=\dots=c_n=2,$$ 
then $\sigma=2n-k-1$ and the integral (\ref{integB}) becomes
$$
J_\lambda(x_1,\dots, x_n)=\int_\gamma \prod_{i=1}^k(x^2-x_i^2)^{1/2}\prod_{j=k+1}^n(x^2-x_j^2)x^{-1} dx.
$$
Taking now small contour $\gamma$ surrounding $x=0$ we have up to a non-essential multiple
$$
J_\lambda(x_1,\dots, x_n)=x_1\dots x_k (x_{k+1}\dots x_n)^2,
$$
which is the remaining potential for the arrangement $D_n^k$ (cf. \cite{OT}).

Note that this case corresponds to the restriction of the Coxeter arrangement of type $D_{k+2n}$ to the subspace
$x_{k+1}=x_{k+n+1}, \, x_{k+2}=x_{k+n+2},\, x_{k+n}=x_{k+2n}.$ So one might expect that the restrictions of Coxeter systems are always harmonic.
This however is not true as the following example shows.

Consider the restriction of the Coxeter system $D_5$ to the subspace $x_3=x_4=x_5.$ The corresponding $\vee$-system $B_3(-1; 1,1,3)$ is of type $D_3(3/2, 3/2)$ in the notations of \cite{FV2} and belongs to the deformation family of the roots of the exceptional Lie superalgebra $D(2,1,\lambda).$

\begin{thm}
The restricted Coxeter $\vee$-system  $B_3(-1; 1,1,3)$  is not harmonic.
\end{thm}

\begin{proof}
We have 7 hyperplanes in the corresponding arrangements. 
Assume that there are polynomial solutions for the corresponding system (\ref{sections}) for
$\kappa_1 \leq \kappa_2 \leq \kappa_3$ with
$$\kappa_1+\kappa_2+\kappa_3=7.$$
Direct check shows that there are no quasi-invariants of degree 3 and the 
space of quasi-invariants of degree 4 is two-dimensional.  We have $\kappa_1\geq 1, \kappa_2\geq 3$, so 
$\kappa_3 \leq 3$ and the only possible choice is
$\kappa_1=1, \kappa_2= \kappa_3=3.$ 
As the  space of quasi-invariants of degree 4 contains the square of the quasi-invariant  of degree 2  one cannot have three independent solutions of the system
\eqref{sections} at the specified $\kappa_i$.
\end{proof}

Note that the corresponding arrangement can be given by
$$
x_3(x_1^2-x_2^2)(x_1^2-x_3^2)(x_2^2-x_3^2)=0
$$
and has Poincar\'e polynomial 
$$P_{\Sigma}(t)=(1+t)(1+3t)^2.$$ 
It is free with a basis of logarithmic vector fields
$$
X_1=\sum_{i=1}^3x_i\partial_i, \, X_2=\sum_{i=1}^3x_i^3\partial_i, \, X_3=x_1x_2x_3^2\sum_{i=1}^3x_i^{-1}\partial_i
$$
(see \cite{OT}, page 251). Note that the restriction of the $D_5$ invariant $x_1\dots x_5$ gives the polynomial
$x_1x_2x_3^3$ of degree 5.

\section{Coxeter arrangements and Saito flat coordinates}

Let $G$ be an irreducible finite Coxeter group generated by reflections in a real Euclidean space $V$ of dimension $n$ and 
$\Delta$ be the set of all corresponding reflection hyperplanes.
Define the corresponding Coxeter root system $\mathcal R$ as a set of normals chosen in a $G$-invariant way. 
Note that we have either 1 or 2 different orbits of $G$ on $\mathcal R,$ 
so such a system in general depends on the additional parameter $q=|\alpha|/|\beta|$, which is a ratio 
of the lengths of the roots from two different orbits.

The positive part $\mathcal A=\mathcal R_+$ of Coxeter root system is known to be a $\vee$-system (\cite{V1}, see also \cite{MG}), 
which we call {\it Coxeter $\vee$-system}. We are going to show that it is harmonic and that the corresponding potentials of the gradient logarithmic vector fields are given by Saito flat coordinates \cite{Saito2}.

Recall briefly the definition of these remarkable coordinates, which can be considered as a canonical choice 
of generators in the algebra of $G$-invariant polynomials $S^G(V).$ 
Let $y_1,\ldots, y_n$ be any set of homogeneous generators in $S^G(V)$ of degrees $d_1>d_2 \geq d_3 \geq \dots >d_n=2.$ 
The image of the Euclidean contravariant metric on $V$ is degenerate on the orbit space $V/G$, but its Lie derivative along well-defined vector field $\frac{\partial}{\partial y_1}$ gives flat metric $\eta$ (called {\it Saito metric}), which is non-degenerate everywhere \cite{Saito, D}. 

The corresponding flat coordinates $t_1, \ldots, t_n \in S^G(V)$ are called {\it Saito flat coordinates}. They were found explicitly by K. Saito et al in \cite{SYS} for all the cases except $E_7, E_8$ (for the latter cases see \cite{N}, \cite{Ab}, \cite{Tal}).
These coordinates play an important role in 2D topological field theory \cite{DVV} and related theory of Frobenius manifolds developed by Dubrovin \cite{D1, D, D2}. In the $A_n$ case they appear in the theory of the dispersionless KP hierarchy \cite{NKNT}.

For the classical Coxeter groups of types $A_n$ and $B_n$ the Saito coordinates can be written as the residues at infinity \cite{IN, DVV}:
$$
t_k = Res_\infty \prod_{i=1}^{n+1} (x-x_i)^{\frac{k}{n+1}}, \quad \sum_{i=1}^{n+1} x_i =0
$$
in type $A_{n}$ and
$$
t_k = Res_\infty \prod_{i=1}^{n} (x^2-x_i^2)^{\frac{2k-1}{2n}} 
$$
in type $B_n.$ 
Comparing this with the formulas (\ref{integ}), (\ref{integB}) we see that they coincide with the potentials of the $A_n(c)$-type $\vee$-systems with $c_0=c_1=\dots = c_n$  and of $B_n(c)$-type $\vee$-systems with $c_0=0, \,c_1=\dots = c_n$ respectively. 

We also note that in the $A_n$ case the Saito coordinate $t_k$ is proportional to the Jack polynomial $J^{\alpha}_{[k+1]},$ corresponding to a single row Young diagram with $k+1$ boxes  and the parameter $\alpha = -\frac{n+1}{k}$ (see \cite{St}).

It turns out that this link with harmonic $\vee$-systems is not accidental.

\begin{thm}
The Coxeter $\vee$-system $\mathcal R_+$ is harmonic. In the case when all the normals have the same length the potentials of the corresponding gradient logarithmic vector fields are the Saito flat coordinates $t_1,\dots, t_n.$
\end{thm}

\begin{proof}
In the case when all the vectors are normalised to have the same length this follows from the results of \cite{FS}, where it was shown that the Saito polynomials satisfy the corresponding system \eqref{EPD} with $\kappa= \deg t_i-1$ being the corresponding exponents of the Coxeter group. 

This covers completely one-orbit cases: simply laced $ADE$ as well as $H_3, H_4$ and odd dihedral groups $I_{2}(2k+1).$ The $B_n$ case follows from Theorem \ref{bnth}:
for a general choice of normals
$$
\mathcal B_n=\{e_i\pm e_j, \,\, \sqrt{2(1+c_0)}e_i, \,\, 1 \leq i < j \leq n\}
$$
the potentials are given by 
$$
F_k = Res_\infty \prod_{i=1}^{n} [x^{2c_0}(x^2-x_i^2)]^{\frac{2k-1}{2n+2c_0}}, \,\, k=1, \dots, n 
$$
(the case of equal lengths corresponds to $c_0=0$).
Thus it remains to consider only the case $F_4$ and even dihedral groups $I_{2}(2p).$

The Coxeter $\vee$-system of type $F_4$ consists of the following covectors:  
$$e_i\pm e_j, \,\, t\sqrt{2}e_i, \,\, \frac{t\sqrt{2}}{2}(e_1\pm e_2 \pm e_3 \pm e_4), 1\leq i < j \leq 4.$$ In the case $t=1$ all the roots have equal length, the case $t=1/\sqrt{2}$ corresponds to the root system $F_4.$
In the complex case we have to add that $t^2\ne -1$ for the non-degeneracy of the corresponding bilinear form. 

Consider the polynomials 
$$
I_n=\sum_{i<j}^4 (x_i-x_j)^n+(x_i+x_j)^n. 
$$
The polynomials $I_2, I_6, I_8, I_{12}$ are basic invariants for the Weyl group of type $F_4$ (see e.g. \cite{SYS}). 
The Mathematica calculations lead to the following potentials   
\begin{gather*}
I_2, 648 (1 + t^2) I_6 - 5 (5 + 4 t^2) I_2^3, \\
 69984 (1 + t^2)^2 I_8 - 9072 (7 + 2 t^2) (1 + t^2) I_2 I_6 + 35 (49 + 46 t^2 + 4 t^4) I_2^4, \\
10077696 (1 + t^2)^3  I_{12}  - 384912 (11 + 8 t^2) (1 + t^2)^2  I_8 I_2^2 + 
 769824 (4 t^2-11) (1 + t^2)^2 I_6^2  \\
+ 7128 (319 + 376 t^2 + 112 t^4) (1 + t^2)  I_6 I_2^3 - 
 11 (3641 + 7032 t^2 + 4560 t^4 + 1048 t^6) I_2^6.
 \end{gather*}
Note that at $t=1$ the above potentials are proportional to the corresponding Saito flat coordinates \cite{SYS} (there seem to be typos in \cite{SYS} in the expressions for the 6th and 12th order polynomials). 

Consider now even dihedral case $I_2(2p)$ with $p>1.$ Let us fix the corresponding vectors as $\a_k=a (\cos \varphi_k, \sin \varphi_k)$, $\b_k=b(\cos \psi_k, \sin \psi_k)$, where $$\varphi_k=\pi/2+\pi k/p, \, \psi_k =\pi/2+ \pi/2p + \pi k/p, \quad k=0,1,\ldots, p-1.$$ One can show that in the complex coordinate $z=x_1+ix_2$  the potentials of the corresponding $\vee$-system are
$$
F_1=z\bar z, \quad F_2=z^{2p} + \bar z^{2p} +\frac{2(2p-1)}{p-1} \frac{a^2-b^2}{a^2+b^2} (z \bar z)^{p}.
$$
Note that when $a=b$ we have the basic invariants $z \bar z$, $z^{2p} + \bar z^{2p}$, known to be Saito flat coordinates in this case \cite{SYS}.
\end{proof}

\section{Concluding remarks}

We have shown that the theory of $\vee$-systems has natural links both with the representation theory of holonomy Lie algebras and with Saito's theory of logarithmic vector fields, which could be used in both ways. This led us to the notion of the harmonic $\vee$-system, which seems to be important in its own right.

The homological representations of braid groups were first studied by Lawrence  \cite{Lawrence} in relation with Hecke algebras and later by Bigelow and Krammer, who showed  that they provide faithful representations of braid groups \cite{Bigelow, Krammer}. The relationship to the monodromy representations of the KZ connections was investigated by Kohno \cite{Kohno3}. Our results can be used to extend these important developments to a larger class of the hyperplane arrangements. 

 We would like to mention a few more very interesting and relevant questions, which were left aside.  The first one is related to the so called ``almost duality" discovered by Dubrovin \cite{D2}. More precisely, Dubrovin found a remarkable connection between the polynomial Frobenius structure on the orbit space of a Coxeter group and the logarithmic Frobenius structure with the prepotential
$$\mathcal F= \sum_{\alpha\in \mathcal R_+} (\alpha,x)^2 \log (\alpha,x)^2,$$
where $\mathcal R$ is the corresponding Coxeter root system with all the roots of the same length \cite{D2}.
A natural question is what is the dual structure in the case when $\mathcal R$ is a general Coxeter $\vee$-system with roots of different lengths, or more generally, if $\mathcal R$ is any harmonic $\vee$-system.

The second question is about possible differential-geometric interpretation of the corresponding potentials.
As we have seen above these potentials are certain deformations of Saito flat coordinates. Finally, it would be interesting to investigate the role of these potentials for the representations  of rational Cherednik algebras similarly to  Saito coordinates which are shown to be related to special singular vectors in the polynomial representations \cite{FS}. 
We hope to address these questions elsewhere.

\section{Acknowledgements}

One of us (APV) is grateful to the Graduate School of Mathematical Sciences of Tokyo University for the hospitality during the summer semester 2014.
He is very grateful to A. Kato, T. Milanov, T. Takebe and especially to T. Kohno and K. Saito for useful and encouraging discussions. MF is grateful to J. Shiraishi for the useful discussion on the relation between Jack and  Saito  polynomials. We thank also M. Pavlov for pointing out the reference \cite{KK}.

This work was partly supported by the EPSRC (grant EP/J00488X/1) and by the Royal Society/RFBR joint project JP101196/11-01-92612.


\begin{thebibliography}{99}



\bibitem{A}
T. Abe, M. Barakat, M. Cuntz, T. Hoge, H. Terao {\it The freeness of ideal subarrangements of Weyl
arrangements.}  arXiv:1304.8033. To appear in JEMS.


\bibitem{Ab}  
D. Abriani {\it Frobenius manifolds associated to Coxeter groups of type E7 and
E8.} arXiv:0910.5453.

\bibitem{Ao}
K. Aomoto {\it On the structure of integrals of power product of linear functions.} Sci. Papers College Gen. Ed. Univ. Tokyo, {\bf 27} (1977), 49-61.

\bibitem{Arnold}
V.I. Arnold {\it The cohomology ring of pure braid group.} Mat. Zametki (Math. Notes), 5:2 (1969), 227-231.

\bibitem{Arsie}
A. Arsie, P. Lorenzoni 
{\it Purely non-local Hamiltonian formalism, Kohno connections and $\vee$-systems.}
J. Math. Phys. {\bf 55} (2014), issue 11, 113510.

\bibitem{Bigelow}
S. Bigelow {\it Braid groups are linear.} J. Amer. Math. Soc. {\bf 14} (2001), 471--486.

\bibitem{Birman}
J. Birman {\it Braids, links, and mapping class groups.}
Annals of Mathematics Studies No. 82, Princeton University Press, Princeton, N.J., 1974.

\bibitem{CV}
O.A. Chalykh, A.P. Veselov {\em Locus configurations and $\vee$-systems} Phys.Lett.A
{\bf 285} (2001), 339--349



\bibitem{CHL}

W. Couwenberg, G. Heckman, E. Looijenga  {\em Geometric structures on the complement of a projective arrangement} Publ. Math. Inst. Hautes Etudes Sci. No. 101 (2005), 69--161. 

\bibitem{DVV}
R. Dijkgraaf, E. Verlinde and H. Verlinde
{\it Topological strings in $D < 1$.} Nucl. Phys. B
{\bf 352} (1991), 59.


\bibitem{D1}
B. Dubrovin
{\em Geometry of 2D topological field theories.} In {\it Integrable Systems and Quantum Groups}, Montecatini, Terme, 1993. Springer Lecture Notes in Math. {\bf 1620} (1996), 120-348.

\bibitem{D}
B. Dubrovin
{\em Differential geometry of the space of orbits of a Coxeter group.}
hep-th/9303152. Surv. Diff. Geom. IV (1999), 213-238.

\bibitem{D2}
B. Dubrovin
{\em On almost duality for Frobenius manifolds.} In {\it Geometry, topology, and mathematical physics},  
AMS Transl. Ser. 2, {\bf 212} (2004), 75-132.

\bibitem{FVimrn}
M.V. Feigin, A.P. Veselov {\it Quasi-invariants and quantum integrals of
the deformed Calogero--Moser systems.} IMRN {\bf 46} (2003), 2487.

\bibitem{FV}
M.V. Feigin, A.P. Veselov {\it Logarithmic Frobenius structures and Coxeter discriminants.} Adv. Math. {\bf 212} (2007), no. 1, 143--162.

\bibitem{FV2}
M.V. Feigin, A.P. Veselov {\it On the geometry of $\vee$-systems.}  
Amer. Math. Soc. Transl. (2) {\bf 224} (2008), 111-123.


\bibitem{F}
M. Feigin {\em On the logarithmic solutions of the WDVV equations.}
Czechoslovak J. Phys. {\bf 56} (2006), no. 10-11, 1149--1153. 


\bibitem{FS}
M. Feigin, A. Silantyev {\em Singular polynomials from orbit spaces.}
Comp. Math. {\bf 148} (2012), 1867-1879.




\bibitem{Humphreys}
J. Humphreys {\it Reflection Groups and Coxeter Groups.} 
Cambridge Univ. Press, 1992.

\bibitem{IN}
S. Ishiura, M. Noumi {\it A calculus of the Gauss-Manin system of type $A_l.$}
Proc. Japan Acad., Ser. A {\bf 58} (1982), 13-16.



\bibitem{KM}
M. Kapovich, J. J. Millson {\it Quantization of bending deformations of polygons in $E^3$, hypergeometric integrals and the Gassner representation.}
Canad. Math. Bull. {\bf 44} (2001), 36-60.




\bibitem{KK}
Y. Kodama, B. Konopelchenko {\it Confluence of hypergeometric functions and integrable hydrodynamic type systems}, PMNP, 2015, arXiv:1510.01540


\bibitem{Kohno} 
T. Kohno {\it On the holonomy Lie algebra and the nilpotent completion of the fundamental group of the complement of hypersurfaces.} Nagoya Math. J. {\bf 92} (1983), 21--37.



\bibitem{Kohno2}
T. Kohno {\it Holonomy Lie algebras, logarithmic connections and the lower central series of fundamental groups.} 
Singularities (Iowa City, IA, 1986), 171--182,
Contemp. Math. {\bf 90}, Amer. Math. Soc., Providence, RI, 1989. 

\bibitem{Kohno3}
T. Kohno {\it Homological representations of braid groups and 
KZ connections.} Journal of Singularities {\bf 5} (2012), 94--108.

\bibitem{Kohno4}
T. Kohno {\it Local systems on configuration spaces, KZ connections and conformal blocks.}
Acta Math. Vietnamica {\bf 39(4)} (2014), 575--598. 

\bibitem{Krammer}
D. Krammer {\it Braid groups are linear.} Ann. of Math. {\bf 155} (2002), 131--156.

\bibitem{Lawrence}
R. J. Lawrence {\it Homological representations of the Hecke algebra.} Comm. Math. Phys. {\bf 135}
(1990), 141--191.

\bibitem{LST} 
O. Lechtenfeld, K. Schwerdtfeger and J. Thueringen
{\it $N=4$ multi-particle mechanics, WDVV equation and roots.} SIGMA {\bf 7} (2011), 023, 21 pages.

\bibitem{L}
E. Looijenga {\it Uniformization by Lauricella functions: an overview of the theory of Deligne-Mostow.} Arithmetic and geometry around hypergeometric functions, 207--244, Progr. Math., 260, Birkh\"auser, Basel, 2007. 

\bibitem{Macd}
I.G. Macdonald 
{\it Symmetric functions and Hall polynomials}, 
Oxford University Press, 1995.

\bibitem{MMM}
A. Marshakov, A. Mironov, and A. Morozov
{\em WDVV-like equations in $N=2$ SUSY Yang-Mills theory.}
Phys. Lett. B, {\bf 389} (1996), 43-52.

\bibitem{MG}
R. Martini, P.K.H. Gragert
{\it Solutions of WDVV equations in Seiberg-Witten theory from root systems.}
J. Nonlin. Math.Phys., {\bf 6} (1) (1999), 1-4.

\bibitem{M}
A.O. Morris {\it Generalizations of the Cauchy and Schur identities}, J. Combinatorial Theory Ser. A, {\bf 11} (1971), 163-169.

\bibitem{NKNT}
T. Nakatsu, A. Kato, M. Noumi and T. Takebe 
{\it Topological strings, matrix integrals, and singularity theory.} 
Phys. Lett. B, {\bf 322} (1994), 192-197.

\bibitem{N}
M. Noumi {\it Expansion of the solutions of a Gauss-Manin system at a point of infinity.} Tokyo J. Math. 7 (1984), no. 1, 1--60. 

\bibitem{OS}
P. Orlik, L. Solomon {\it Coxeter arrangements.}
Proc. Symp. Pure Math. {\bf 40} (1983), Part 2, 269-291. 

\bibitem{OT}
P. Orlik, H. Terao {\it Arrangements of Hyperplanes.} Springer Verlag, 1991.

\bibitem{OT1}
P. Orlik, H. Terao {\it Arrangements and Hypergeometric Integrals.} MSJ Memoirs, {\bf 9} (2001).

\bibitem{OT2}
P. Orlik, H. Terao {\it Coxeter arrangements are hereditarily free.} 
Tohoku Math J. {\bf 45} (1993), 369-383.

\bibitem{Oxley}
J. Oxley {\it Matroid Theory.} Oxford Graduate Texts in Mathematics, Oxford, 1993.

\bibitem{Pavlov}

M. Pavlov {\it  Integrable hydrodynamic chains}, J. Math. Phys. 44 (2003), no. 9, 4134--4156. 

\bibitem{Saito}
K. Saito {\it Theory of logarithmic differential forms and logarithmic vector fields.} 
J. Fac. Sci. Univ. Tokyo Sect. IA Math. {\bf 27} (1980), 265-291.

\bibitem{Saito2}
K. Saito  {\it On a linear structure of the quotient variety by a finite
reflection group.} Preprint RIMS-288 (1979). Publ. RIMS, Kyoto Univ. {\bf 29} (1993), 535-579.

\bibitem{Saito3}
K. Saito {\it Uniformization of the orbifold of a finite reflection group.} Frobenius manifolds, 265-320, 
Aspects Math., E36, Vieweg, Wiesbaden, 2004.  

\bibitem{SYS}
 K. Saito, T. Yano, J. Sekiguchi {\it On a certain generator system of the ring of
invariants of a finite reflection   group}, Comm. in Algebra {\bf 8}(4) (1980), 373--
408.

\bibitem{SchV}
V. Schreiber, A.P. Veselov {\it On deformation and classification of $\vee$-systems.}
J. Nonlin. Math. Phys. {\bf 21}(4) (2014), 543-583.


\bibitem{Sergan}
V. Serganova {\it On generalization of root systems.} Commun. in
Algebra {\bf 24} (1996), 4281-4299.

\bibitem{SV}
A.N. Sergeev, A.P. Veselov {\it Deformed quantum Calogero-Moser
systems and Lie superalgebras.} Comm. Math. Phys. {\bf 245} (2004),
249--278.

\bibitem{St}
R.P. Stanley {\it Some combinatorial properties of Jack symmetric functions}, Advances in Math, {\bf 77} (1989), 76--115.


\bibitem{Tal}
 V. Talamini {\it Flat bases of invariant polynomials and $\hat P$-matrices of $E_7$ and $E_8$},
J. Math. Phys. 51 (2010), no. 2, 023520, 20 pp., arXiv:1003.1095.

\bibitem{Terao}
H. Terao {\it Generalized exponents of a free arrangement of hyperplanes and Shepherd-Todd-Brieskorn formula.}
Invent. Math. {\bf 63} (1981), 159-179.


\bibitem{V1}
A.P. Veselov {\em Deformations of root systems and new solutions to generalised
WDVV equations.} Phys. Lett. A 261 (1999), 297.

\bibitem{V2}
A.P. Veselov {\em On geometry of a special class of solutions to generalised
WDVV equations. } hep-th/0105020.  In: {\it Integrability: the Seiberg-Witten and Whitham equations} 
(Edinburgh, 1998), Gordon and Breach (2000), 125--135.

\bibitem{V3}
A.P. Veselov {\it On generalisations of Calogero-Moser-Sutherland quantum problem and WDVV equations.}  J. Math. Phys. {\bf 43} (2002), 5675-82.

\bibitem{Yuzvinsky}
S. A. Yuzvinskii {\it Orlik--Solomon Algebras in Algebra and Topology.}
 Uspekhi Mat. Nauk 56 (2001), no. 2(338), 87--166; translation in Russian Math. Surveys 56 (2001), no. 2, 293--364.
 
 \bibitem{Zaslavsky}
 T. Zaslavsky {\it The geometry of root systems and signed graphs.} Amer. Math. Monthly {\bf 88} (1981), 88-105.

\end{thebibliography}
\end{document}